\documentclass[11pt]{amsart}

\usepackage{epsfig}
\usepackage{amsmath, amsthm, amssymb}
\usepackage{color}

\newtheorem{defn}{Definition}[section]

\newtheorem{thm}[defn]{Theorem}
\newtheorem{prop}[defn]{Proposition}
\newtheorem{lemma}[defn]{Lemma}

\newtheorem{remark}[defn]{Remark}

\newcommand{\be}{\begin{equation}}
\newcommand{\ee}{\end{equation}}
\newcommand{\bea}{\begin{eqnarray}}
\newcommand{\eea}{\end{eqnarray}}
\newcommand{\beas}{\begin{eqnarray*}}
\newcommand{\eeas}{\end{eqnarray*}}

\newcommand{\R}{\mathbb{R}}

\newcommand{\E}{\mathbb{E}}

\newcommand{\ve}{\varepsilon}

\newcommand{\goto}{\rightarrow}
\newcommand{\hsp}{\hspace{.3in}}
\newcommand{\bp}{\begin{proof}}
\newcommand{\ep}{\end{proof}}

\newcommand{\TV}{\tiny{\mbox{TV}}}

\newcommand{\dstyle}{\displaystyle}

\begin{document}

\title[Bipartite Potts Model]{The Aggregate Path Coupling Method for the Potts Model on Bipartite Graph}

\author{Jos\'{e} C. Hern\'{a}ndez}
\address{University of Campinas, Institute of Mathematics, Statistics and Scientific Computation, Rua Sergio Buarque de Holanda 651, 13083-850, Campinas, SP, Brazil}
\email{josecehe@gmail.com}
\thanks{JCH would like to thank the FAPESP foundation for the financial support, project 2015/16407-8}

\author{Yevgeniy Kovchegov}
\address{Department of Mathematics, Oregon State University, Corvallis, OR  97331}
\email{kovchegy@math.oregonstate.edu}
\thanks{YK was supported in part by the NSF DMS-1412557 award.}

\author{Peter T. Otto}
\address{Department of Mathematics, Willamette University, Salem, OR 97302,  phone: 1-503-370-6487}
\email{potto@willamette.edu}


\begin{abstract}
In this paper, we derive the large deviation principle for the Potts model on the complete bipartite graph $K_{n,n}$ as $n$ increases to infinity. Next, for the Potts model on $K_{n,n}$, we provide an extension of the method of aggregate path coupling that was originally developed in \cite{KOT} for the mean-field Blume-Capel model and in \cite{KO} for a general mean-field setting that included the Generalized Curie-Weiss-Potts model analyzed in  \cite{JKRW}. We use the aggregate path coupling method to identify and determine the threshold value $\beta_s$ separating the rapid and slow mixing regimes for the Glauber dynamics of the Potts model on $K_{n,n}$.

\end{abstract}

\date{\today}

\maketitle

\section{Introduction} 

\medskip

The study of mixing times of dynamics of statistical mechanical models continues to be an active area of research, not only for its practical applications in sampling and simulations, but also because of the theoretical connections with the corresponding equilibrium phase transition structure.  Much of recent work investigating this connection has been limited to mean-field models which can be viewed as models defined on the complete graph.  These include the Curie-Weiss model \cite{LPW}, which is the Ising model on the complete graph, and the mean-field Blume-Capel model \cite{KOT}.  In these models, the macroscopic quantity is the one dimensional magnetization.  The research has also extended to the Curie-Weiss-Potts \cite{CDLLPS}, the Potts model on the complete graph, and the Generalized Curie-Weiss-Potts model \cite{KO} for which the macroscopic quantity is the higher dimensional empirical measure.  

A common phenomenon shown in these mean-field results is that for models that exhibit a continuous phase transition with respect to temperature at equilibrium, the mixing time of the Glauber dynamics also undergoes a transition from rapid to slow mixing at precisely the same critical value as the equilibrium phase transition.  On the other hand, for models that exhibit a discontinuous phase transition, the mixing time transition of the corresponding dynamics occurs at a value that differs from the equilibrium phase transition critical value.  More specifically, the mixing time transition from rapid to slow mixing occurs at the onset of metastable states of the model at equilibrium.

The equilibrium phase transition structure of these mean-field models have been obtained by deriving the large deviation principle (LDP) for the macroscopic quantities \cite{CET, EOT, JKRW} and using the upper large deviations bound to define the equilibrium states in terms of the rate function of the LDP.  To prove mixing time results of the corresponding Glauber dynamics, the standard path coupling method derived by Bubley and Dyer \cite{BD} has shown to be quite effective in determining the temperature parameter regime where rapid mixing occurs for models that undergo continuous phase transitions \cite{LPW}.  But for models that undergo discontinuous phase transitions, the standard path coupling method fails because one no longer has contraction of the mean coupling distance for every possible pairs of values of the macroscopic quantity.  For these models (which also applies to the continuous phase transition case), the method of aggregate path coupling was developed \cite{KOsurvey}  to determine the parameter regime of rapid mixing of Glauber dynamics and has been successfully applied to mean-field models that exhibit discontinuous phase transitions \cite{KOT, KO}.

In this paper, we illustrate the strength and generality of the aggregate path coupling method for proving rapid mixing by applying it to the Potts model on the bipartite graph which differs from the traditional mean-field model where every spin interacts with every other spin; i.e. interactions are defined by the complete graph.  Recent studies include \cite{BGG, C, CF, CG, FU, GT} that describe the dynamics and equilibrium structure of related models of interaction between two or more families of particles, including as a particular case the Ising/Potts model on general bipartite graphs. The contributions of this paper first include the large deviation principle for the Potts model on the bipartite graph that yields the equilibrium phase transition structure of the model and then identifying the interface value $\beta_s$ at which the Glauber dynamics exhibits rapid mixing for $\beta < \beta_s$ using the method of aggregate path coupling.  The (somewhat surprising) result obtained in Lemma \ref{beta_sq} is that the interface value $\beta_s$ for the Potts model on the bipartite graph $K_{n,n}$ is equal to the corresponding value for the Curie-Weiss-Potts model, which is the Potts model on the complete graph \cite{CDLLPS}.

The bipartite graphs are often used in the study of complex networks such as cellular and metabolic networks. They naturally arise in the statistical mechanics of complex networks. See \cite{AB, Strogatz}. The bipartite graphs are also used in the study of human disease networks \cite{GCVCVB} and in molecular biology \cite{HCLFG}. 

The paper is organized as follows. In Section \ref{bcwp} the Potts model on the bipartite graph $K_{n,n}$ is described. In Section \ref{em} the large deviation principle for the Potts model on $K_{n,n}$ is obtained, and the complete description of the equilibrium macrostates is provided for all values of the parameter $\beta$. In Section \ref{sec:gd} we introduce the Glauber dynamics on the set of configurations of the Potts model on $K_{n,n}$. At the end of  Section \ref{sec:gd}, the main result that identifies the exact parameter region of rapid mixing is stated in Theorem  \ref{rap_mixBCWP}.  Next, a greedy coupling of the Glauber dynamics is introduced for the  Potts model on the bipartite graph $K_{n,n}$ in Section \ref{coupling}. Finally, the main result Theorem  \ref{rap_mixBCWP} is proved in Section \ref{aggpath} via the method of aggregate path coupling that was introduced in \cite{KOT} and \cite{KO}.

\section{The Potts model on the Bipartite Graph} \label{bcwp} 

In this section we introduce the Potts model on the bipartite graph $K_{n,n}$ which is defined by the Gibbs measure $P_{n,n,\beta}$ in terms of the Hamiltonian energy function. Then we rewrite $P_{n,n,\beta}$ in terms of 
its magnetization vectors on each side of the bipartite graph in order to prove the large deviation principle. Finally, we define the free energy function  $\psi(\beta)$.

\medskip
Let $q$ be a fixed integer and define the collection of vectors $\Lambda = \{ e^1, e^2, \ldots, e^q \}$, where $e^k$ are the $q$ standard basis vectors of $\R^q$ referred to as possible spin values in statistical mechanical models.  Let $\rho=\frac{1}{q}\sum\limits_{i=1}^q \delta_{e^i}$ and let $P_n= \prod_{j=1}^n \rho$ denote the product measure on
$\Lambda^n$ with one-dimensional marginals $\rho$. Thus $P_n\{\sigma\}=1/q^n$ for each $\sigma\in\Lambda^n$. We
also denote by $\rho$ the probability vector in $\R^q$ all of whose coordinates equal $q^{-1}$.

A {\it configuration} of a model on the bipartite graph $K_{n,n}$ has the form $(\sigma,\tau) \in \Lambda^n \times \Lambda^n$, where the spin configuration on the left set of $n$ vertices of $K_{n,n}$ is denoted by $\sigma$ and the spin configuration on the right set of $n$ vertices is denoted by $\tau$.  The Hamiltonian for the Potts model on the bipartite graph $K_{n,n}$ is defined by 
$$H_n(\sigma,\tau)=-\frac{1}{n}\sum_{i=1}^n\sum_{j=1}^n\delta(\sigma_i,\tau_j),$$
where $\delta(u,v)=\begin{cases}
      1 & \text{ if } u=v \\
      0 & \text{ if } u \not=v
\end{cases}$.
Note that with the Hamiltonian defined as above, the interactions of the model are governed by the edges of the bipartite graph $K_{n,n}$; more specifically, the spin values of $\sigma$ on the left side of the bipartite graph $K_{n,n}$ only interact with the spin values of $\tau$ on the right side of $K_{n,n}$.

The Potts model on the bipartite graph or the bipartite Potts model (BPM) is defined by
the probability of $(\sigma,\tau) \in \Lambda^n \times \Lambda^n$, corresponding to inverse temperature 
$\beta>0$ given by the canonical ensemble or Gibbs measure
\begin{equation}\label{ensemble_BCWP}
P_{n,n,\beta}(\sigma,\tau)=\frac{1}{Z_{n,n}(\beta)}\exp(-\beta H_n(\sigma,\tau))P_n\!\times\!P_n(\sigma,\tau)
\end{equation}
where $Z_{n,n}(\beta)$ is the partition function
$$Z_{n,n}(\beta)=\int_{\Lambda^n\times\Lambda^n}\exp(-\beta H_n(\sigma,\tau))d P_n\!\times\!P_n(\sigma,\tau)=
\sum_{\sigma,\tau\in\Lambda^n}\exp(-\beta H_n(\sigma,\tau))\frac{1}{q^{2n}}.$$

In terms of the microscopic quantities, the spins at each vertex of $K_{n,n}$, the relevant macroscopic quantity is the pair of 
{\it magnetization vectors} (empirical measures or proportion vectors) $\big(L_n(\sigma), L_n(\tau)\big)\in \R^q\times\R^q$ with
\be
L_n(\omega) = \big(L_{n,1}(\omega), L_{n,2}(\omega), \ldots, L_{n,q}(\omega)\big), 
\ee
where the $k$th component is defined by 
\[ L_{n, k}(\omega) = \frac{1}{n} \sum_{i=1}^n \delta(\omega_i, e^k) \]
which yields the proportion of spins in $\omega \in \Lambda^n$ that take on the value $e^k$.  The magnetization vector $L_n$ takes values in the set of probability vectors 
\be
\label{eqn:latticesimplex} 
\mathcal{P}_n = \left\{ \left(\frac{n_1}{n}, \frac{n_2}{n}, \ldots, \frac{n_q}{n} \right) : \mbox{each} \ n_k \in \{0, 1, \ldots, n\} \ \mbox{and} \ \sum_{k=1}^q n_k = n \right\} 
\ee
which lies inside the continuous simplex
$$ \mathcal{P}_q = \left\{ \nu \in \R^q : \nu = (\nu_1, \nu_2, \ldots, \nu_q), \mbox{each} \ \nu_k \geq 0, \ \sum_{k=1}^q \nu_k = 1 \right\}.$$

\noindent
Let $\langle\cdot, \cdot\rangle$ denote the inner product on $\R^q$. Then, since 
$$
\begin{array}{ccl}
\langle L_n(\sigma),L_n(\tau)\rangle &=& \displaystyle\sum_{k=1}^q L_{n,k}(\sigma)L_{n,k}(\tau)\\
          &=& \displaystyle\sum_{k=1}^q \left(\frac{1}{n} \sum_{i=1}^n \delta(\sigma_i,e^k) \right) 
                                        \left(\frac{1}{n} \sum_{j=1}^n \delta(\tau_j,e^k) \right)\\
          &=& \displaystyle\frac{1}{n^2}\sum_{i,j=1}^n\sum_{k=1}^q \delta(\sigma_i,e^k)\delta(\tau_j,e^k)
          =\displaystyle\frac{1}{n^2}\sum_{i,j=1}^n \delta(\sigma_i,\tau_j),       
\end{array}
$$
it follows that the Hamiltonian for the bipartite Potts model can be rewritten as 
$$H_n(\sigma,\tau)=-n\langle L_n(\sigma),L_n(\tau)\rangle.$$
Hence,
$$P_{n,n,\beta}(\sigma,\tau)=\frac{1}{Z_{n,n}(\beta)}
\exp\left[n\beta\langle L_n(\sigma),L_n(\tau)\rangle\right] P_n\!\times\!P_n(\sigma,\tau),$$
where
$$Z_{n,n}(\beta)=\int_{\Lambda^n\times\Lambda^n}\exp\left[n\beta\langle L_n(\sigma),L_n(\tau)\rangle\right]d P_n\!\times\!P_n(\sigma,\tau).$$

\medskip
\noindent
The above expression of $P_{n,n,\beta}$ allows us to define the {\bf interaction representation function} for the bipartite Potts model $H: \R^q\times\R^q\to \R$ as follows
\be
\label{IRF}
H(x,y)=-\langle x,y \rangle=- x_1y_1-x_2y_2- \cdots-x_qy_q.
\ee
This function is a finite  $\mathcal{C}^{\infty}(\R^q\times\R^q)$ function satisfying
$$H_n(\sigma,\tau)=nH(L_n(\sigma),L_n(\tau)).$$
Utilizing the interaction representation function $H$, the bipartite Potts model can be expressed as 
$$P_{n,n,\beta}(B)=
\frac{1}{Z_{n,n}(\beta)}\int_B\exp\left[-\beta n H(L_n(\sigma),L_n(\tau))\right] dP_n\times dP_n,$$
where $P_n$ is the product measure with identical marginals $\rho$, $B$ belong to the $\sigma$-field of subsets of 
$\Lambda^n\times\Lambda^n$, and 
$$Z_{n,n}(\beta)=\int_{\Lambda^n\times\Lambda^n}\exp\left[-\beta n H(L_n(\sigma),L_n(\tau))\right] dP_n\times dP_n.$$

\medskip
\noindent
The free energy for the model is the quantity $\psi(\beta)$ defined by the limit
$$
\begin{array}{ccl}
-\beta\psi(\beta)&=&\lim_{n\to\infty}\frac{1}{2n}\log Z_{n,n}(\beta).
\end{array}
$$

\bigskip

\section{Equilibrium Macrostates of the Bipartite Potts Model} \label{em} 

Following the approach developed in \cite{EHT}, the equilibrium phase structure of the bipartite Potts model will be defined by the large deviation principle satisfied by $P_{n,n,\beta}$.  
By Sanov's Theorem with respect to the probability measure $P_n$, the $q$-dimensional magnetization vector $L_n$ satisfies
the large deviation principle on $\mathcal{P}_q$ with rate function expressed as the relative entropy 
\[ R(\nu | \rho) = \sum_{k=1}^q \nu_k \log \left( \frac{\nu_k}{\rho_k} \right), \]
where we let $\rho$ denote the probability vector in $\R^q$ all of whose coordinates equal $q^{-1}$.

\medskip
\noindent
Therefore, the $2q$-dimensional magnetization vector $(L_n,L_n)$ satisfies the large deviation principle with respect to the product measure $P_n\!\times\!P_n$ over 
$\mathcal{P}_q\times\mathcal{P}_q$ with rate function given by the sum of relative entropies, that is,
\begin{equation}
P_n\!\times\!P_n((L_n,L_n)\in d\gamma\times d\nu)\approx e^{-n(R(\gamma | \rho)+R(\nu | \rho))}.
\end{equation}

\medskip
\noindent
Denote $R((\gamma,\nu)|\rho)=R(\gamma | \rho)+R(\nu | \rho)$. Now, since $(L_n,L_n)$ satisfies the large deviation principle on 
$\mathcal{P}_q\times\mathcal{P}_q$ with respect to $P_n\!\times\!P_n$ with rate function $R((\gamma,\nu)|\rho)$, the Laplace principle implies the following lemma (see \cite{Ellis} and \cite{Varadhan}).

\begin{lemma}\label{free_E} 
		The free energy for the model satisfies
		\be 
		-\beta\psi(\beta) = \sup_{(\gamma, \nu)\in\mathcal{P}_q\times\mathcal{P}_q} \alpha_\beta(\gamma,\nu).
		\ee
		where 
		\begin{equation}
		\alpha_\beta(\gamma,\nu)=\beta\langle \gamma, \nu \rangle - R((\gamma,\nu)|\rho)  
		\end{equation}
	\end{lemma}
	
\medskip
\noindent
Next, applying Lemma \ref{free_E}, we obtain the following large deviation principle for the Potts model on $K_{n,n}$.

\begin{thm}\label{LDP_Knn}
The empirical vector pair $(L_n,L_n)$ satisfies 
the large deviation principle with respect to the canonical ensemble probability measure $P_{n,n,\beta}$, as defined in (\ref{ensemble_BCWP}), on $\mathcal{P}_q\times\mathcal{P}_q$ 
with the rate function
$$I_{\beta}(\gamma,\nu)=R((\gamma,\nu) |\rho) +\beta H(\gamma,\nu)-
\inf_{\gamma^{\prime},\nu^{\prime}} \{ R((\gamma^{\prime},\nu^{\prime}) |\rho) +\beta H(\gamma^{\prime},\nu^{\prime})\},$$
where $H$ is the interaction representation function defined in (\ref{IRF}). Specifically, for any closed subset $F$ of $\mathcal{P}_q\times\mathcal{P}_q$,
\be
\label{eqn:ldpbound}
 \limsup_{n \goto \infty} \frac{1}{n} \log P_{n,n, \beta} ( (L_n, L_n) \in F ) \leq - \inf_{(\gamma,\nu) \in F} I_\beta ((\gamma,\nu)) 
 \ee
 and for any open subset $G$ of $\mathcal{P}_q\times\mathcal{P}_q$,
 \[ 
  \liminf_{n \goto \infty} \frac{1}{n} \log P_{n, n, \beta} ( (L_n, L_n) \in G ) \geq - \inf_{(\gamma,\nu) \in G} I_\beta ((\gamma,\nu)).
\]  
\end{thm}
	
\medskip
\noindent
The LDP upper bound (\ref{eqn:ldpbound}) stated in the above theorem yields the following natural definition of the {\bf equilibrium macrostates} for the bipartite Potts model
$$
\mathcal{E}_\beta ~=~ \left\{ (\gamma,\nu) \in \mathcal{P}_q\times\mathcal{P}_q :  I_{\beta}(\gamma,\nu)=0 \right\}
~=~ \left\{ (\gamma,\nu) \in \mathcal{P}_q\times\mathcal{P}_q :  (\gamma,\nu) \ \mbox{maximizes} \ \alpha_\beta(\gamma,\nu) \right\}.
$$

\medskip
\begin{remark}
In \cite{KO} it is assumed that the interaction representation function $H$ is concave in order to guarantee that the free energy functional $G_\beta$ defined in (\ref{FEF}) has a maximum. In the case of the bipartite model, $H$ is neither convex nor concave over $\mathcal{P}_q\times\mathcal{P}_q$, but the properties of $H$ are sufficient to guarantee the existence of a unique macrostate in the single phase region.
\end{remark}
	
\medskip
\noindent
We analyze the equilibrium macrostates $\mathcal{E}_\beta$  by writing $\alpha_\beta(\gamma,\nu)$ as follows
\be\label{alpha_func}
\alpha_\beta(\gamma,\nu)=\left( \frac{\beta}{2}\langle \gamma,\gamma\rangle-R(\gamma|\rho)  \right) +
\left( \frac{\beta}{2}\langle \nu,\nu\rangle-R(\nu|\rho) \right)-\frac{\beta}{2}\|\gamma-\nu \|^2,
\ee
and arriving to the following lemma.

\begin{lemma}
\label{maxidentity}
The maximum of $\alpha_\beta(\gamma,\nu)$ occurs on the identity line $\gamma=\nu$.
\end{lemma}
\bp
Consider the function $f(z)=\frac{\beta}{2}\langle z, z\rangle-R(z|\rho)$ over the compact set $\mathcal{P}_q$. Denote by $K_\beta$ the set of global maximum points of the function $f$ on $\mathcal{P}_q$. In \cite{EW},  Ellis and Wang  show that 
 $$K_\beta=\left\{
 \begin{array}{lcl}
 \{ \rho \} &\mbox{for}& 0<\beta<\beta_c,\\
 \{ z^{1}(\beta),\dots,  z^{q}(\beta)\} &\mbox{for}& \beta>\beta_c,\\
 \{ \rho, z^{1}(\beta_c),\dots,  z^{q}(\beta_c)\} &\mbox{for}& \beta=\beta_c,
 \end{array} 
 \right.$$
 where, as before, $\rho=\left( \frac{1}{q} , \dots, \frac{1}{q}\right)$.
 For $\beta\geq \beta_c$, the points in $K_\beta$ are all distinct, and $$\beta_c=\frac{2(q-1)}{q-2}\log(q-1).$$
 	
\medskip
\noindent
 Now, for a given $z\in K_\beta$, Theorem 2.1 in \cite{EW} implies that  
 $$\alpha_\beta (\gamma,\nu) \leq \alpha_\beta (z,z),\quad\mbox{for all}\quad (\gamma,\nu)\in \mathcal{P}_q\times\mathcal{P}_q.$$
 Hence, the maximum of $\alpha_\beta(\gamma,\nu)$ occurs on the identity line $\gamma=\nu$.
\ep

\medskip
\noindent
Following Lemma \ref{maxidentity}, in order to compute the equilibrium macrostates of the bipartite Potts model we need to minimize the function $ -\alpha_{\beta}(\gamma,\nu)$ restricted to the set 
$(\mathcal{P}_q\times\mathcal{P}_q) \cap \{(\gamma,\nu)\in\R^q\times\R^q : \gamma=\nu\}$. Thus, we rewrite the 
set $\mathcal{E}_\beta$ as follow
$$\mathcal{E}_\beta= \left\{ (\gamma,\gamma) \in \mathcal{P}_q\times\mathcal{P}_q : 
\gamma \ \mbox{minimizes} \ R(\gamma| \rho) - \frac{\beta}{2}\langle \gamma,\gamma\rangle \right\}.$$
Borrowing from the notations in \cite{EW}, we denote $\alpha_\beta(\gamma)=\frac{\beta}{2}\langle \gamma,\gamma\rangle - R(\gamma| \rho)$. 
The function $\alpha_\beta(\cdot)$ was used to study the Curie-Weiss-Potts model,  
and the global minima of $R(\gamma| \rho) - \frac{\beta}{2}\langle \gamma,\gamma\rangle=-\alpha_{\beta}(\gamma)$ were obtained in \cite{EW} and \cite{Wu}.
Hence, the corresponding result describing the structure of the set $\mathcal{E}_\beta$ for the bipartite Potts model follows.

\medskip
\noindent
Define the function $\varphi:[0,1]\to \mathcal{P}_q$ as follows 
$$\varphi(s)=(q^{-1}[1+(q-1)s], q^{-1}(1-s),\dots,q^{-1}(1-s)).$$

\begin{thm}\label{LD_BCWP}
Fix a positive integer $q\geq 3$. Let $\beta_c=(2(q-1)/(q-2))\log(q-1)$, and for $\beta>0$ let $s(\beta)$ denote the 
largest root of the equation
$$s=\frac{1-e^{\beta s}}{1+(q-1)e^{-\beta s}}.$$
The following conclusions hold.
\begin{enumerate}
 \item[(a)] The quantity $s(\beta)$ is well defined. It is positive, strictly increasing, and differentiable with respect to 
 $\beta$ in an open interval containing $[\beta_c,\infty)$. Also, $s(\beta_c)=(q-2)/(q-1)$, and $\lim_{\beta\to\infty}s(\beta)=1$.
 \item[(b)] For $\beta\geq\beta_c$, define $\nu^{1}=\varphi(s(\beta))$ and let $\nu^{i}$ $~(i=1,2,\dots,q)~$ denote the points in $\R^{q}$ obtained by interchanging the first and the $i^{th}$ coordinates of $\nu^{1}$. Then
 $$\mathcal{E}_\beta=\left\{
\begin{array}{ccl}
\{ (\rho,\rho) \} &\mbox{for}& 0<\beta<\beta_c,\\
\{ (\nu^{1},\nu^{1}),\dots,  (\nu^{q},\nu^{q})\} &\mbox{for}& \beta>\beta_c,\\
\{ (\rho,\rho), (\nu^{1},\nu^{1}),\dots,  (\nu^{q},\nu^{q})\} &\mbox{for}& \beta=\beta_c
\end{array} 
\right.$$
For all $\beta\geq\beta_c$, the points in $\mathcal{E}_\beta$ are all distinct. The point $\nu^{1}(\beta_c)$  equals $\varphi(s(\beta_c))=\varphi((q-2)/(q-1))$.
\end{enumerate}
\end{thm}

The behavior exhibited by the set of equilibrium macrostates $\mathcal{E}_\beta$ for the bipartite Potts model stated in Theorem \ref{LD_BCWP} is referred to as a {\it discontinuous, first-order phase transition} with respect to the parameter $\beta$.  This is because as $\beta$ passes through the critical value $\beta_c$ from below, in the set of equilibrium macrostates $\mathcal{E}_\beta$, a spontaneous emergence of additional {\it distinct} macrostates occurs.  The Potts model on the complete graph also exhibits a discontinuous, first-order phase transition.  As will be shown in Section \ref{aggpath}, it was for this type of equilibrium phase transition for which the classical path coupling method failed and for which the aggregate path coupling method was developed to solve. See \cite{KOT} and \cite{KO}.

An important quantity, called the free energy functional, is defined below in terms of the interaction representation function $H$ and the logarithmic moment generating function. The {\bf logarithmic moment generating function} for the bipartite Potts model is 
\be\label{lgfKnn}
\Gamma(x,y)=\log\left( \frac{1}{q}\sum_{i=1}^{q} e^{x_i}\right) + \log\left( \frac{1}{q}\sum_{i=1}^{q} e^{y_i}\right)
\ee
and the {\bf free energy functional} for the canonical ensemble $P_{n,n,\beta}$ is 
\be
\label{FEF}
G_{\beta}(x,y) =\beta(-H)^{*}(-\nabla H(x,y)) - \Gamma(-\beta\nabla H(x,y)),
\ee
where $H^\ast$ denotes the Legendre-Fenchel transform defined below. As  $H(x,y)=-\langle x,y\rangle$, we have $\nabla H(x,y)=-(y_1,\dots,y_q, x_1,\dots,x_q)$ and 
$$
\begin{array}{ccl}
(-H)^{*}(x,y) &=& \sup_{(z,w)\in \R^{q}\times\R^{q}}\{\langle (z,w),(x,y) \rangle + H(z,w)\}\\
 &=&  \sup_{(z,w)\in \R^{q}\times\R^{q}}\{\langle z,x \rangle + \langle w,y \rangle - \langle z,w\rangle \}\\
 &=& \langle x,y\rangle.
\end{array}
$$
Therefore
$$
\begin{array}{ccl}
G_{\beta}(x,y) &=& \beta(-H)^{*}(-\nabla H(x,y)) - \Gamma(-\beta\nabla H(x,y))\\
&=&  \beta (-H)^{*}(y,x)-\Gamma (\beta x,\beta y)\\
&=& \beta \langle x,y\rangle - \log\left( \frac{1}{q}\displaystyle\sum_{i=1}^{q} e^{\beta x_i}\right) - \log\left( \frac{1}{q}\displaystyle\sum_{i=1}^{q} e^{\beta y_i}\right).
\end{array}
$$
Next, employing the identity $2\langle x, y\rangle= \| x\|^{2}+\| y\|^{2} -\| x-y\|^{2}$, and defining 
$$G_{\beta}(x)=\frac{\beta}{2}\langle x,x\rangle - \log\left( \frac{1}{q}\displaystyle\sum_{i=1}^{q} \exp\{\beta x_i\}\right),$$ 
for all $x\in\R^{q}$, we rewrite the function  $G_{\beta}(x,y) $ as follow
\be
G_{\beta}(x,y) =G_{\beta}(x)+G_{\beta}(y)-\frac{\beta}{2}\| x-y\|^{2}.  
\ee
Then, by Theorem A.1  in \cite{CET} (or a more general version stated in Lemma 3.3 of \cite{KO}) we have  that 
$$\sup_{(x,y)\in\mathcal{P}_q\times\mathcal{P}_q}\{ \alpha_{\beta}(x,y) \}=\inf_{(x,y)\in\R^{q}\times\R^{q}}\{ G_{\beta}(x,y) \}.$$

\bigskip

\section{Glauber Dynamics and Mixing Times} \label{sec:gd}

Below, we define the Glauber dynamics for the bipartite Potts model over the configuration space $\Lambda^n \times \Lambda^n$.  These dynamics are governed by a reversible Markov chain with stationary distribution $P_{n,n, \beta}$ defined in (\ref{ensemble_BCWP}). There, on each time step we perform the following 
\begin{description}
  \item[i] Select a vertex $i$ on $K_{n,n}$ uniformly;
  \item[ii] Update the spin at vertex $i$ according to the distribution $P_{n, n, \beta}$, conditioned on the event that the spins at all vertices not equal to $i$ remain unchanged.
\end{description}

\medskip
\noindent
Specifically, for any given configuration $\omega=(\omega_1,\dots,\omega_n) \in\Lambda^{n}$, denote by $\omega_{i,e^{k}}$ the configuration that agrees with $\omega$ at all vertices $j\neq i$ and the spin at vertex $i$ is $e^{k}$; i.e.
$$\omega_{i,e^{k}}=(\omega_1,\dots,\omega_{i-1},e^{k},\omega_{i+1},\dots,\omega_n).$$
Then, on the bipartite graph $K_{n,n}$, if the current configuration is $(\sigma,\tau)$, there are two possible update probabilities for the Glauber dynamics of the bipartite Potts model.  These are 
\be\label{gdleft}
P((\sigma,\tau)\to (\sigma_{i,e^{k}},\tau))=\frac{e^{-\beta n H(L_n(\sigma_{i,e^{k}}), L_n(\tau))}}{\displaystyle\sum_{l=1}^{q}e^{-\beta n H(L_n(\sigma_{i,e^{l}}), L_n(\tau))}}
\ee
and 
\be\label{gdright}
P((\sigma,\tau)\to (\sigma,\tau_{i,e^{k}}))=\frac{e^{-\beta n H(L_n(\sigma), L_n(\tau_{i,e^{k}}))}}{\displaystyle\sum_{l=1}^{q}e^{-\beta n H(L_n(\sigma), L_n(\tau_{i,e^{l}}))}}
\ee
Note that vertex $i$, on the left or right side of $K_{n,n}$, is selected uniformly.

\medskip
\noindent
Now, as it was done in \cite{KO}, we show that the update probabilities of the Glauber dynamics introduced above can be expressed in terms of the derivative of the logarithmic moment generating function 
$\Gamma$ defined in (\ref{lgfKnn}). For our analysis we introduce the following two functions 
\be\label{g_left}
g_{x_l}^{H,\beta}(x,y)=[\partial_{x_l}\Gamma]\big(-\beta\nabla H(x,y)\big)= \frac{e^{\beta y_l}}{\displaystyle\sum_{k=1}^{q}e^{\beta y_k}}
\ee
and
\be\label{g_right}
g_{y_l}^{H,\beta}(x,y)=[\partial_{y_l}\Gamma]\big(-\beta\nabla H(x,y)\big)= \frac{e^{\beta x_l}}{\displaystyle\sum_{k=1}^{q}e^{\beta x_k}}
\ee
Next, we make an important observation that will be used later in the paper. We notice that $g_{x_l}^{H,\beta}(x,y)$ only depends on $y$ and $g_{y_l}^{H,\beta}(x,y)$ only depends on $x$.

\medskip
\noindent
Also we define $g^{H,\beta}:~\mathcal{P}_q\times\mathcal{P}_q \rightarrow \mathcal{P}_q\times\mathcal{P}_q$ as follows:
$$g^{H,\beta}(x,y)=\big((g_{x_1}^{H,\beta},\dots,g_{x_q}^{H,\beta}), (g_{y_1}^{H,\beta},\dots,g_{y_q}^{H,\beta})\big).$$
Our next result is the following lemma.

\begin{lemma}\label{prob_trans}
Let $P\big((\sigma,\tau)\to (\sigma_{i,e^{k}},\tau)\big)$ and  $P\big((\sigma,\tau)\to (\sigma,\tau_{i,e^{k}})\big)$ be the Glauber dynamics update probabilities given in 
(\ref{gdleft}) and  (\ref{gdright}), respectively. Then, for any $k\in\{1,\dots,q\}$,
$$P\big((\sigma,\tau)\to (\sigma_{i,e^{k}},\tau)\big)=g_{x_k}^{H,\beta}\big(L_n(\sigma),L_n(\tau)\big)  + O\left(\frac{1}{n^{2}}\right),$$
and 
$$P\big((\sigma,\tau)\to (\sigma,\tau_{i,e^{k}})\big)=g_{y_k}^{H,\beta}\big(L_n(\sigma),L_n(\tau)\big) + O\left(\frac{1}{n^{2}}\right).$$
\end{lemma}
\begin{proof}
Suppose $\sigma_i=e^{m}$ and the update $(\sigma,\tau)\to (\sigma_{i,e^{k}},\tau)$ is chosen. Given the interaction representation function $H(x,y)=-\langle x,y\rangle$, we have that its gradient and Hessian matrix  are given by 
$$\nabla H (x,y)=-(y_1,\dots,y_q,x_1,\dots,x_q),$$ 
and 
$$
\mbox{Hess}(H)=\left(
\begin{array}{cccc|cccc}
0 & 0 & \cdots & 0 &-1 & 0 & \cdots& 0\\
0 & 0 & \cdots & 0 & 0 & -1 & \cdots& 0\\
\vdots&\vdots&\ddots& \vdots&\vdots&\vdots&\ddots&\vdots\\
0 & 0 & \cdots & 0 & 0 & 0 & \cdots& -1\\
\hline
-1 & 0 & \cdots & 0 & 0 & 0 & \cdots& 0\\
0 & -1 & \cdots & 0 & 0 & 0 & \cdots& 0\\
\vdots&\vdots&\ddots& \vdots&\vdots&\vdots&\ddots&\vdots\\
0 & 0 & \cdots & -1 & 0 & 0 & \cdots& 0\\
\end{array}
\right)_{2q\times 2q}
$$
Applying Taylor's theorem to the function $H$ we have
$$
\begin{array}{ccl}
H(L_n(\sigma_{i,e^{k}}),L_n(\tau)) &=& H(L_n(\sigma),L_n(\tau)) + \displaystyle\sum_{l=1}^{q}\frac{\partial H}{\partial x_l}(L_n(\sigma),L_n(\tau)) \left[L_{n,l}(\sigma_{i,e^{k}})-L_{n,l}(\sigma)\right] +\\
 &&\displaystyle\frac{1}{2} ( L_{n}(\sigma_{i,e^{k}})-L_{n}(\sigma),0,\dots,0)^{T} \mbox{Hess}(H)( L_{n}(\sigma_{i,e^{k}})-L_{n}(\sigma),0,\dots,0) +o\left(\frac{1}{n^{2}}\right)\\
&=& H(L_n(\sigma),L_n(\tau)) + \displaystyle\sum_{l=1}^{q} -L_{n,l}(\tau) \left[L_{n,l}(\sigma_{i,e^{k}})-L_{n,l}(\sigma)\right] + o\left(\frac{1}{n^{2}}\right).
\end{array}
$$
Now, note that 
$$L_{n,l}(\sigma_{i,e^{k}})-L_{n,l}(\sigma)=\frac{1}{n}(\delta(e^{k},e^{l})- \delta(\sigma_i,e^{l})).$$
Thus, since $\sigma_i=e^{m}$,
\beas
\sum_{l=1}^{q} -L_{n,l}(\tau) \left[L_{n,l}(\sigma_{i,e^{k}})-L_{n,l}(\sigma)\right]  & = & \frac{1}{n} \big(-L_{n,k}(\tau)+L_{n,m}(\tau)\big) \\
& = & \frac{1}{n}\left( \frac{\partial H}{\partial x_k}(L_n(\sigma),L_n(\tau)) -  \frac{\partial H}{\partial x_m}(L_n(\sigma),L_n(\tau))\right). 
\eeas
Therefore, we have that
$$H(L_n(\sigma_{i,e^{k}}),L_n(\tau)) = H(L_n(\sigma),L_n(\tau)) +\frac{1}{n}\left( \frac{\partial H}{\partial x_k}(L_n(\sigma),L_n(\tau)) -  \frac{\partial H}{\partial x_m}(L_n(\sigma),L_n(\tau))\right) +  o\left(\frac{1}{n^{2}}\right).$$
Similarly, for the update $(\sigma,\tau)\to (\sigma,\tau_{i,e^{k}})$, if $\tau_i=e^{m}$, we have that  
$$H(L_n(\sigma),L_n(\tau_{i,e^{k}})) = H(L_n(\sigma),L_n(\tau)) +\frac{1}{n}\left( \frac{\partial H}{\partial y_k}(L_n(\sigma),L_n(\tau)) -  \frac{\partial H}{\partial y_m}(L_n(\sigma),L_n(\tau))\right) +  o\left(\frac{1}{n^{2}}\right).$$
The above two expressions imply that the transition probabilities (\ref{gdleft}) and 
(\ref{gdright}) can be expressed as
$$P((\sigma,\tau)\to (\sigma_{i,e^{k}},\tau))=g_{x_k}^{H,\beta}(L_n(\sigma),L_n(\tau))  +O\left(\frac{1}{n^{2}}\right)$$
and 
$$P((\sigma,\tau)\to (\sigma,\tau_{i,e^{k}}))=g_{y_k}^{H,\beta}(L_n(\sigma),L_n(\tau)) + O\left(\frac{1}{n^{2}}\right)$$
respectively.
\end{proof}

\medskip
Now, as it was observed following formula (\ref{g_right}), the function $g_{x_k}^{H,\beta}(x,y)$ depends only on $y$  and $g_{y_k}^{H,\beta}(x,y)$  depends only on $x$. 
Consequently, it is convenient to introduce the following function $g^{H,\beta}(z):~\mathcal{P}_q \rightarrow \mathcal{P}_q$, for $z\in\mathcal{P}_q$, defined as 
$$g^{H,\beta}(z) = \big( g_{1}^{H,\beta}(z), g_{2}^{H,\beta}(z), \ldots, g_{q}^{H,\beta}(z) \big) \hsp \mbox{where} \hsp  g_{k}^{H,\beta}(z)=\frac{e^{\beta z_k}}{\displaystyle\sum_{j=1}^{q}e^{\beta z_j}}.$$
Then, $$g^{H,\beta}(x,y)=\big((g_{x_1}^{H,\beta},\dots,g_{x_q}^{H,\beta}), (g_{y_1}^{H,\beta},\dots,g_{y_q}^{H,\beta})\big)=\big(g^{H,\beta}(y),g^{H,\beta}(x)\big).$$
Utilizing this new notation for $g_{x_k}^{H,\beta}$ and $g_{y_k}^{H,\beta}$, we rewrite  the probability transitions as follows:
$$P((\sigma,\tau)\to (\sigma_{i,e^{k}},\tau))=g_{k}^{H,\beta}(L_n(\tau))  +O\left(\frac{1}{n^{2}}\right),$$
and 
$$P((\sigma,\tau)\to (\sigma,\tau_{i,e^{k}}))=g_{k}^{H,\beta}(L_n(\sigma)) + O\left(\frac{1}{n^{2}}\right).$$  
This new expression emphasizes the fact that the probability transition on the left side depend  on the right configuration in the bipartite graph $K_{n,n}$, and vice versa.

\medskip

The mixing time measures the rate of convergence of a Markov chain to its stationary distribution and is defined in terms of 
the total variation distance. The {\it total variation distance} between two distributions $\mu$ and $\nu$ on the configuration space $\Omega$ is defined by
\[ \| \mu - \nu \|_{\TV} = \sup_{A \subset \Omega} | \mu(A) - \nu(A)| = \frac{1}{2} \sum_{x \in \Omega} | \mu(x) - \nu(x)|. \]
Given a convergent Markov chain, we define the {\it maximal distance to its stationary distribution} to be 
\[ d(t) = \max_{x \in \Omega} \| P^t(x, \cdot) - \pi \|_{\TV} \]
where $P^t(x, \cdot)$ is the probability distribution at time $t$ of the Markov chain originating at $x \in \Omega$, and $\pi$ is its stationary distribution.
Then, given $\epsilon>0$, the {\bf mixing time} of a Markov chain is defined as
\[ t_{mix}(\epsilon) = \min\{ t : d(t) \leq \epsilon \} \smallskip .\]
Finally, it is sometimes more convenient to bound
the {\it standardized maximal distance} defined by 
\be
\label{eqn:StanMaxDist}
\bar{d}(t) := \max_{x,y \in \Omega} \|P^t(x, \cdot) - P^t(y, \cdot) \|_{\TV} 
\ee
which satisfies the following inequality
\be
\label{bbar}
d(t) \leq \bar{d}(t) \leq 2 \, d(t). 
\ee
See \cite{LPW} for a survey on the theory of mixing times. 

\bigskip
\noindent
Rates of mixing times are generally categorized into two groups: {\it rapid mixing} which implies that the mixing time exhibits polynomial growth with respect to the system size $n$, and {\it slow mixing} which implies that the mixing time grows exponentially with the system size.  
The rapid mixing region for the bipartite Potts model is determined by the following parameter value:
{\small
\be
\label{mixcritical}
\beta_s(q)=\sup\left\{\beta\geq 0 : g_k^{H,\beta}(x)<y_k\;\mbox{and}\; g_k^{H,\beta}(y)<x_k\;\mbox{for all}\; 
(x,y)\in\mathcal{P}_q\times \mathcal{P}_q \;\mbox{such that}\; x_k,y_k\in\left(\frac{1}{q},1\right] \right\}
\ee
}

\begin{lemma}\label{beta_sq}
If $\beta_c(q)$ is the critical  value defined in Theorem \ref{LD_BCWP}, then
$$\beta_s(q)\leq \beta_c(q).$$
\end{lemma}
\bp
Recall the corresponding $\beta_s$ value for the Curie-Weiss-Potts model as derived in \cite{CDLLPS},
$$\beta_s^{CWP}(q)=\sup\left\{\beta\geq 0 : g_k^{H,\beta}(x)<x_k\;\mbox{ for all }\; x \in\mathcal{P}_q \;\mbox{ such that }\; x_k\in\left(\frac{1}{q},1\right] \right\}.$$
Also, in \cite{CDLLPS}, the inequality $~\beta_s^{CWP}(q)<\beta_c(q)$ was proved, where $\beta_c(q)$ is the same for the Curie-Weiss-Potts model as for the bipartite Potts model, as shown in Theorem \ref{LD_BCWP}.

\bigskip
\noindent
Next, we prove that $\beta_s(q)=\beta_s^{CWP}(q)$. We partition the values of $\beta$ into the following three subsets,
{\small
$$B^-=\left\{\beta\geq 0 : g_k^{H,\beta}(x)<y_k\;\mbox{and}\; g_k^{H,\beta}(y)<x_k\;\forall 
(x,y)\in\mathcal{P}_q\times \mathcal{P}_q \;\mbox{such that}\; x_k,y_k\in\left(\frac{1}{q},1\right], y_k<x_k \right\},$$
$$B^+=\left\{\beta\geq 0 : g_k^{H,\beta}(x)<y_k\;\mbox{and}\; g_k^{H,\beta}(y)<x_k\;\forall 
(x,y)\in\mathcal{P}_q\times \mathcal{P}_q \;\mbox{such that}\; x_k,y_k\in\left(\frac{1}{q},1\right], y_k>x_k \right\},$$
and
$$B^0=\left\{\beta\geq 0 : g_k^{H,\beta}(x)<y_k\;\mbox{and}\; g_k^{H,\beta}(y)<x_k\;\forall  
(x,y)\in\mathcal{P}_q\times \mathcal{P}_q \;\mbox{such that}\; x_k,y_k\in\left(\frac{1}{q},1\right], y_k=x_k \right\},$$
}
and note that $\sup_{\beta}B^-=\sup_{\beta}B^+\leq \sup_{\beta}B^0=\beta_s^{CWP}(q)\leq \beta_s(q)$. Furthermore, we have that 
$$B^-\cup B^0 \cup B^+ \subseteq \left\{\beta\geq 0 : g_k^{H,\beta}(z)<z_k\;\mbox{for all}\; 
z\in\mathcal{P}_q\;\mbox{such that}\; z_k\in\left(\frac{1}{q},1\right] \right\}.$$
Thus $\beta_s(q)\leq \beta_s^{CWP}(q)$. This concludes the proof of the Lemma \ref{beta_sq}.
\ep

\medskip
Determining the parameter regime where a model undergoes rapid mixing is of major importance, as it is in this region that the application of the Glauber dynamics is physically feasible.  This rapid mixing parameter regime is the main result of this paper and we state it next.

\begin{thm}\label{rap_mixBCWP} Let $\beta_s(q)$ be as defined in formula (\ref{mixcritical}).  Then for $\beta< \beta_s(q)$, the mixing time of the Glauber dynamics for the bipartite 
Potts model satisfies
$$t_{mix}(\epsilon) =O(n\log n).$$
\end{thm}
Finally, the standard argument using the Cheeger constant (see \cite{LPW}) shows slow mixing for $\beta > \beta_s(q)$. Thus, the above result provides the boundary point for the rapid mixing region. 

\bigskip

\section{Coupling of Glauber Dynamics for the Bipartite Potts model} \label{coupling} 

We begin by recalling the definition of a discrepancy distance for a pair of configurations $\omega$ and $\omega'$ in $\Lambda^n$ as used in \cite{KO}:
$$d(\omega,\omega')= \sum_{j=1}^{n}1_{\{\omega_j \neq\omega'_j\}}.$$
Next, we define a similar metric on the configuration space $\Lambda^n\times\Lambda^n$. For a pair of configurations, $(\sigma,\tau)$ and  $(\sigma',\tau')$  in $\Lambda^n\times\Lambda^n$, 
we define the distance between them as
$$
\begin{array}{ccl}
d\big((\sigma,\tau),(\sigma',\tau')\big)&=&\displaystyle\sum_{j=1}^{n}1_{\{\sigma(j)\neq\sigma'(j)\}}+\sum_{j=1}^{n}1_{\{\tau(j)\neq\tau'(j)\}},\\
  & & \\
  &=& d(\sigma,\sigma') + d(\tau,\tau').
\end{array}
$$  

\medskip
\noindent
Let $X_t=(X_t^1,X_t^2)$ and  $Y_t=(Y_t^1,Y_t^2)$  be two copies of the Glauber dynamics of the bipartite Potts model. Here, we consider the standard greedy coupling of $X_t$ and $Y_t$. A greedy coupling is an efficient and easy-to-implement coupling construction that often yields 
an optimal order of upper bound on the mixing time. See \cite{AF, HC}.
Specifically, at each time step, a vertex is selected at random, uniformly from the $2n$ vertices. Let $i$ denote the selected vertex.
Next, we update the $i$-th spin to the same value in both processes with the largest possible probability. 
Suppose that $X_t=(\sigma,\tau)$ and  $Y_t=(\sigma',\tau')$. 

\medskip
\noindent
If $i$ belongs to the left side of the bipartite graph $K_{n,n}$, we define
$$p_{1,k}=P\big((\sigma,\tau)\to (\sigma_{i,e^{k}},\tau)\big), \quad\mbox{and}\quad 
q_{1,k}=P\big((\sigma',\tau')\to (\sigma'_{i,e^{k}},\tau')\big).$$ 
We also let $~P_{1,k}^{left}=\min\{p_{1,k}, q_{1,k}\}~$ and $~P_1=\sum_{k=1}^{q} P_{1,k}^{left}$.

\medskip
\noindent
If $i$ belongs to the right side of the graph $K_{n,n}$, we define
$$p_{2,k}=P\big((\sigma,\tau)\to (\sigma,\tau_{i,e^{k}})\big), \quad\mbox{and}\quad q_{2,k}=P\big((\sigma',\tau')\to (\sigma',\tau'_{i,e^{k}})\big).$$
We also let $~P_{2,k}^{right}=\min\{p_{2,k}, q_{2,k}\}~$ and $~P_2=\sum_{k=1}^{q} P_{2,k}^{right}$.

\medskip
\noindent
Once the vertex $i$ is selected, we delete the spin at $i$ in both processes, and replace it with a new one. The joint transition probabilities for the coupled process are set as follows.
\begin{description}
  \item[i] For $i$ belonging to the left half of the graph $K_{n,n}$, we  let for each $k \in \{1,\hdots,q\}$,
  $$P\big(X_{t+1}=(\sigma_{i,e^k},\tau),~Y_{t+1}=(\sigma'_{i,e^k},\tau') ~\big|~ X_t=(\sigma,\tau),~Y_t=(\sigma',\tau')\big) = {1 \over 2n} P_{1,k}^{left}.$$
  \item[ii] For $i$ belonging to the left half of the graph $K_{n,n}$, we  let for each pair of spins $k,m  \in \{1,\hdots,q\}$,
  $$P\big(X_{t+1}=(\sigma_{i,e^k},\tau),~Y_{t+1}=(\sigma'_{i,e^m},\tau') ~\big|~ X_t=(\sigma,\tau),~Y_t=(\sigma',\tau')\big) = {(p_{1,k}-P_{1,k}^{left})(q_{1,m}-P_{1,m}^{left}) \over 2n(1-P_1)}.$$
  Observe that the above probability is zero if $k=m$.
  \item[iii] For $i$ belonging to the right half of the graph $K_{n,n}$, we  let for each $k \in \{1,\hdots,q\}$,
  $$P\big(X_{t+1}=(\sigma,\tau_{i,e^k}),~Y_{t+1}=(\sigma',\tau'_{i,e^k}) ~\big|~ X_t=(\sigma,\tau),~Y_t=(\sigma',\tau')\big) = {1 \over 2n} P_{2,k}^{right}.$$
  \item[iv] For $i$ belonging to the right half of the graph $K_{n,n}$, we  let for each pair of spins $k,m  \in \{1,\hdots,q\}$,
  $$P\big(X_{t+1}=(\sigma,\tau_{i,e^k}),~Y_{t+1}=(\sigma',\tau'_{i,e^m}) ~\big|~ X_t=(\sigma,\tau),~Y_t=(\sigma',\tau')\big) = {(p_{2,k}-P_{2,k}^{right})(q_{2,m}-P_{2,m}^{right}) \over 2n(1-P_2)}.$$
\end{description}

\medskip
\noindent
Observe that once $X_t=Y_t$, the processes remain matched (coupled) for the rest of the time. In the coupling literature, the time
$$\tau_c=\min\{t \geq 0 : X_t=Y_t \}$$
is referred to as the {\it coupling time}.

\medskip
\noindent
For a coupling of a Markov chain $(X_t, Y_t)$, the mean coupling distance  $~\mathbb{E}[d(X_t, Y_t) ]~$ is tied to the total variation distance (and thus the mixing time) via the following inequality known as  the {\it coupling inequality} \cite{LPW}: 
\be \label{coupling_ineq} 
\| P^t(x, \cdot) - P^t(y, \cdot) \|_{{\scriptsize \mbox{TV}}} \leq P(X_t \neq Y_t)  \leq \mathbb{E}[d(X_t, Y_t) ] 
\ee
The above inequality and (\ref{bbar}) imply that the order of the mean coupling distance is an upper bound on the order of the mixing time. We next derive the mean coupling distance for the bipartite Potts model.

\medskip
\noindent
Let $\mathcal{I}_1=\{i_1,\dots,i_{d_{left}}\}$ and $\mathcal{I}_2=\{j_1,\dots,j_{d_{right}}\}$ be the sets of vertices at which the spin values of the two initial configuration $(\sigma,\tau)$ and  $(\sigma',\tau')$ disagree. 
Define $\kappa(e^l)$ to be the probability that the coupled processes update differently when the chosen vertex $j\notin \mathcal{I}_1\cup\mathcal{I}_2$ has spin $e^l$. 

Recall that in the greedy coupling, we update the $j$-th spin to the same value $e^{k}$ in both processes with the largest possible probability. Thus, the probability of the $j$-th spin updating to $e^k$ in exactly one of the two process (but not in the other) should be either 
$$\left| P\big((\sigma,\tau)\to (\sigma_{i,e^{k}},\tau)\big)-P\big((\sigma',\tau')\to (\sigma'_{i,e^{k}},\tau')\big) \right|$$
or
$$\left|P\big((\sigma,\tau)\to (\sigma,\tau_{i,e^{k}})\big)-P\big((\sigma',\tau')\to (\sigma',\tau'_{i,e^{k}})\big) \right|$$
depending on what side of the bipartite graph vertex $j$ was selected from. Thus, if the chosen vertex $j$ is such that it has spin value $e^l$, then by Lemma \ref{prob_trans},
\begin{align}\label{eq:k}
\kappa(e^l) = & \frac{1}{2}\sum_{k=1}^{q} \left| P\big((\sigma,\tau)\to (\sigma_{i,e^{k}},\tau)\big)-P\big((\sigma',\tau')\to (\sigma'_{i,e^{k}},\tau')\big) \right|  \nonumber \\
&+\frac{1}{2}\sum_{k=1}^{q} \left|P\big((\sigma,\tau)\to (\sigma,\tau_{i,e^{k}})\big)-P\big((\sigma',\tau')\to (\sigma',\tau'_{i,e^{k}})\big) \right| \nonumber \\
= & \frac{1}{2}\sum_{k=1}^{q} \left|g_{k}^{H,\beta}(L_n(\tau')) - g_{k}^{H,\beta}(L_n(\tau)) \right| + \frac{1}{2}\sum_{k=1}^{q} \left|g_{k}^{H,\beta}(L_n(\sigma')) - g_{k}^{H,\beta}(L_n(\sigma)) \right| + O\left( \frac{1}{n^2} \right).
\end{align}
Note that the coefficient $\frac{1}{2}$ in the above equation corrects the double counting.
Also, observe that the choice of $j$ and therefore of its spin $e^l$ would change the value of $\kappa$ in (\ref{eq:k}) by a magnitude of order $O\left( \frac{1}{n^2} \right)$, which is incremental for this computation. 

\medskip
\noindent
Therefore, as $g^{H,\beta}:\mathcal{P}_q\to \R$ is in $\mathcal{C}^2$, for all $n$ large enough there exists $C'>0$ such that 
{\small
\[
\left| \kappa(e^l) - \frac{1}{2}\sum_{k=1}^{q} \left| \left\langle L_n(\tau') - L_n(\tau), \nabla g_{k}^{H,\beta}(L_n(\tau))\right\rangle \right| - \frac{1}{2}\sum_{k=1}^{q} \left| \left\langle L_n(\sigma') - L_n(\sigma), \nabla g_{k}^{H,\beta}(L_n(\sigma))\right\rangle \right|   \right| <C^\prime\varepsilon^2,
\] 
}
for $\varepsilon \leq \parallel L_n(\sigma,\tau) - L_n(\sigma',\tau') \parallel <2\varepsilon$ and all values of $l\in\{1,2,\dots,q\}$. 

\medskip
\noindent
Denote 
$$\kappa_{left}=\frac{1}{2}\sum_{k=1}^{q} \left| \left\langle L_n(\tau') - L_n(\tau), \nabla g_{k}^{H,\beta}(L_n(\tau))\right\rangle \right|$$
and
$$\kappa_{right}=\frac{1}{2}\sum_{k=1}^{q} \left| \left\langle L_n(\sigma') - L_n(\sigma), \nabla g_{k}^{H,\beta}(L_n(\sigma))\right\rangle \right|.$$
Then, the mean coupling distance after one iteration of the coupling process starting in $(\sigma,\tau)$ and $(\sigma',\tau')$ 
is bounded as follows
{\small
\begin{align}\label{dist_1step}
\E_{\substack{ (\sigma,\tau)\\(\sigma',\tau')}} [ d(X,Y) ] & \leq \left(1-{1 \over 2n}\right)d\big((\sigma,\tau),(\sigma',\tau')\big)+{1 \over 2}\kappa_{left}+{1 \over 2}\kappa_{right}+c\ve^2 \nonumber \\
 & \leq d\big((\sigma,\tau),(\sigma',\tau')\big)\left[ 1-\frac{1}{2n}\left( 1-\frac{\kappa_{left} +\kappa_{right}+ 2c\varepsilon^2}{d\big((\sigma,\tau),(\sigma',\tau')\big)/n} \right) \right]
\end{align}
}
for a fixed $c>0$ and all $n$ large enough, whereas the original distance was
$$d\big((\sigma,\tau),(\sigma',\tau')\big)= d(\sigma,\sigma') + d(\tau,\tau').$$

\medskip
\noindent
The contraction of the mean coupling distance between {\it all} pairs of configurations is often difficult to prove. Consequently, the method of path coupling was introduced by Bubley and Dyer \cite{BD} which expressed the mean coupling distance between arbitrary configurations in terms of the mean coupling distance between {\it neighboring} configurations with respect to some path metric.  
The path coupling method was then generalized for the case when the mean coupling distance does not contract for at least some of the neighboring configurations. See \cite{KOT} and \cite{KO}. This generalization is referred to as {\it aggregate path coupling} and was first developed in \cite{KOT} and then extended in \cite{KO}.  In those two papers, the models studied were classical mean-field models of statistical mechanics which can be viewed as spin models defined on the complete graph $K_n$.  In the current paper, we apply the aggregate path coupling method for the first time beyond $K_n$. 

\bigskip

\section{Aggregate Path Coupling  for the Bipartite Potts model} \label{aggpath} 

We begin with an overview of the aggregate path coupling method for bounding the mixing time of the Glauber dynamics of the bipartite Potts model.  First, the measure concentration result of the large deviation upper bound (\ref{eqn:ldpbound}) makes it sufficient to show contraction of the mean coupling distance between a coupled process where one starts in an arbitrary configuration and the other starts in a configuration for which the macroscopic quantity for the bipartite Potts model, the pair of empirical vectors, is near equilibrium; i.e. $\parallel \big(L_n(\sigma),L_n(\tau)\big) - (\rho, \rho)\parallel_1 < \varepsilon'$.  The proof to this first part is given in general in Theorem 9.2 of \cite{KO}. 

Second, in order to prove contraction of the mean coupling distance of a coupling of the Glauber dynamics of the bipartite Potts model where one starts near equilibrium, we aggregate the intermediate distances over a monotone path in the configuration space $\Lambda^n\times\Lambda^n$ defined below. The aggregation over the discrete path is carried out by integrating over an approximating continuous path in the continuous space $\mathcal{P}_q\times\mathcal{P}_q$. The details of this second step are provided next.

\medskip 

Let $(\sigma,\tau)$ and $(\sigma',\tau')$ be configurations in $\Lambda^n\times\Lambda^n$. Consider a path $\pi$ in $\Lambda^n\times\Lambda^n$ connecting configurations $(\sigma,\tau)$ and $(\sigma',\tau')$,
$$\pi : (\sigma,\tau)=(x_0^1,x_0^2),(x_1^1,x_1^2),\dots,(x_r^1,x_r^2)=(\sigma',\tau').$$

\begin{defn}\label{mono_path}
We say that $\pi$ is a {\it monotone path} if 
\begin{enumerate}
	\item[(i)] $\dstyle\sum_{i=1}^{r}d\big((x_{i-1}^1,x_{i-1}^2),(x_{i}^1,x_{i}^2)\big)=d\big((\sigma,\tau),(\sigma',\tau')\big)$;
	\item[(ii)] for each $k \in \{1,2,\hdots,q\}$ and $j\in \{1,2\}$, the $k$th coordinate of $L_n(x_i^j)$, denoted by $L_{n,k}(x_i^j)$ is monotonic as $i$ increases from $0$ to $r$.
\end{enumerate}
\end{defn}

\medskip
\noindent
Now, fix $\varepsilon>0$ and suppose $\pi : (\sigma,\tau)=(x_0^1,x_0^2),(x_1^1,x_1^2),\dots,(x_r^1,x_r^2)=(\sigma',\tau')$ is a monotone path connecting  $(\sigma,\tau)$ and $(\sigma',\tau')$ such that for each $i \in \{1,\hdots,r\}$,
\begin{equation}\label{ineq:eps}
\varepsilon\leq \big\| \big(L_n(x_i^1), L_n(x_i^2)\big)-\big(L_n(x_{i-1}^1), L_n(x_{i-1}^2)\big) \big\|_1 <2\varepsilon .
\end{equation} 
Equation (\ref{dist_1step}) implies the following bound on the mean distance for a coupling process starting in  configurations $(\sigma,\tau)$ and $(\sigma',\tau')$:
\begin{align}\label{bound_dist}
\E_{\substack{ (\sigma,\tau)\\(\sigma',\tau')}} [d(X,Y)] 
& \leq  \sum_{i=1}^r \mathbb{E}_{\substack{ (x_{i-1}^1,x_{i-1}^2)\\ (x_i^1,x_i^2)}}[d(X_{i-1}, X_i)]   \nonumber \\
 & \leq   \sum\limits_{i=1}^r \left\{ d\big((x_{i-1}^1,x_{i-1}^2),(x_i^1,x_i^2)\big) \cdot \left[1 -{1 \over 2n} \left(1-{\kappa_{1,i}+\kappa_{2,i}+2c \varepsilon^2 \over d\big((x_{i-1}^1,x_{i-1}^2),(x_i^1,x_i^2)\big)/n} \right)  \right]\right\}  \nonumber \\ 
& =  d\big((\sigma,\tau),(\sigma',\tau')\big) \left[ 1-{1 \over 2n}\left(1-{S_1+S_2 + 4c\varepsilon^2 \over  2d\big((\sigma,\tau),(\sigma',\tau')\big)/n }\right) \right]  \nonumber \\ 
&\leq d\big((\sigma,\tau),(\sigma',\tau')\big)\left[ 1-\dstyle\frac{1}{2n}\left( 1-\frac{ S_1+S_2 + 4c\varepsilon^2}{\parallel L_n(\sigma) - L_n(\sigma') \parallel_1 + \parallel L_n(\tau) - L_n(\tau') \parallel_1} \right) \right],
\end{align}
where 
$$\kappa_{1,i}= {1 \over 2}\sum\limits_{k=1}^q \Big| \Big<L_n(x_i^1) - L_n(x_{i-1}^1), \nabla g_k^{H, \beta}(L_n( x_{i-1}^1)) \Big> \Big|,$$
$$\kappa_{2,i}={1 \over 2}\sum\limits_{k=1}^q \Big| \Big<L_n(x_i^2) - L_n(x_{i-1}^2), \nabla g_k^{H, \beta}(L_n( x_{i-1}^2)) \Big> \Big|,$$
$$ S_1=2\dstyle\sum_{i=1}^{r} \kappa_{1,i} =\dstyle\sum_{k=1}^{q}\dstyle\sum_{i=1}^{r}\left| \left\langle L_n(x_i^1)-L_n(x_{i-1}^1),\nabla g_k^{H,\beta}(L_n(x_{i-1}^1)) \right\rangle \right|,$$
and
$$S_2=2\dstyle\sum_{i=1}^{r} \kappa_{2,i} = \dstyle\sum_{k=1}^{q}\dstyle\sum_{i=1}^{r}\left| \left\langle L_n(x_i^2)-L_n(x_{i-1}^2),\nabla g_k^{H,\beta}(L_n(x_{i-1}^2)) \right\rangle \right|.$$
Notice that  (\ref{ineq:eps}) allows applying (\ref{dist_1step}) for every link $\big((x_{i-1}^1,x_{i-1}^2),(x_i^1,x_i^2)\big)$ of the path $\pi$. Thus, (\ref{bound_dist}) is obtained by aggregating (\ref{dist_1step}) along the path $\pi$.

\bigskip
\noindent
 {\bf Aggregate path coupling heuristics.} The classical path coupling relies on showing contraction along any monotone path connecting two configurations, in one time step. Here we observe that we only need to show contraction along \underline{one} monotone path connecting two configurations in order to have the mean coupling distance $\E_{\substack{ (\sigma,\tau)\\(\sigma',\tau')}} [d(X,Y)]$ contract in a single time step. However, finding even one monotone path with which we can show contraction in the equation (\ref{bound_dist}) is not easy. The answer to this is in finding a continuous monotone path $(\gamma, \widetilde{\gamma})$ in $\mathcal{P}_q \times \mathcal{P}_q$ connecting $\big(L_n(\sigma),L_n(\tau)\big)$ and  $\big(L_n(\sigma'),L_n(\tau')\big)$, such that 
\begin{equation}\label{bound_dist2}
{ \sum\limits_{k=1}^{q}\int\limits_\gamma\left|\langle \nabla g_{k}^{H,\beta}(x),dx\rangle \right | ~+~ \sum\limits_{k=1}^{q}\int\limits_{\widetilde{\gamma}}\left|\langle \nabla g_{k}^{H,\beta}(y),dy\rangle \right | \over  \parallel L_n(\sigma) - L_n(\sigma') \parallel_1 + \parallel L_n(\tau) - L_n(\tau') \parallel_1 } ~<1 
\end{equation}
Although  $(\gamma, \widetilde{\gamma})$  is a continuous  path in continuous space $\mathcal{P}_q \times \mathcal{P}_q$, it serves in finding a monotone path 
$$\pi : (\sigma,\tau)=(x_0^1,x_0^2),(x_1^1,x_1^2),\dots,(x_r^1,x_r^2)=(\sigma',\tau')$$
in $\Lambda^n\times\Lambda^n$ connecting configurations $(\sigma,\tau)$ and $(\sigma',\tau')$
such that 
$$\big(L_n(x_0^1),L_n(x_0^2)\big),~\big(L_n(x_1^1),L_n(x_1^2)\big),\dots,~\big(L_n(x_r^1),L_n(x_r^2)\big)$$
in $\mathcal{P}_q \times \mathcal{P}_q$ are positioned along $(\gamma, \widetilde{\gamma})$ and satisfy (\ref{ineq:eps}). The quantities 
$S_1$ and $S_2$ defined in (\ref{bound_dist}) are the Riemann sums approximating $~\sum\limits_{k=1}^{q}\int\limits_\gamma\left|\langle \nabla g_{k}^{H,\beta}(x),dx\rangle \right | ~$ and $~ \sum\limits_{k=1}^{q}\int\limits_{\widetilde{\gamma}}\left|\langle \nabla g_{k}^{H,\beta}(y),dy\rangle \right |$
respectively. Therefore we obtain
$${S_1+S_2+ 4c\varepsilon^2  \over  \parallel L_n(\sigma) - L_n(\sigma') \parallel_1 + \parallel L_n(\tau) - L_n(\tau') \parallel_1 }  ~<1,$$
 for $\ve$ small enough and $n$ large enough. This will imply contraction of the mean coupling distance $\E_{\substack{ (\sigma,\tau)\\(\sigma',\tau')}} [d(X,Y)]$ in (\ref{bound_dist}). 

\medskip
\noindent
The above inequality (\ref{bound_dist2}) motivates the definition of the aggregate $g$-variation between a pair of points $(x',y')$ and $(x'',y'')$ in $\mathcal{P}_q\times\mathcal{P}_q$ along a continuous monotone path $(\gamma, \widetilde{\gamma})$ defined as follows
\begin{align*}
D_{(\gamma, \widetilde{\gamma})}^{g}((x',y'),(x'',y'')) &= \sum\limits_{k=1}^{q}\int_\gamma\left|\langle \nabla g_{k}^{H,\beta}(x),dx\rangle \right | ~+~ \sum\limits_{k=1}^{q}\int_{\widetilde{\gamma}}\left|\langle \nabla g_{k}^{H,\beta}(y),dy\rangle \right |\\
&= D_\gamma^{g}(x',x'') ~+~ D_{\widetilde{\gamma}}^{g}(y',y''),
\end{align*}
where $~D_\gamma^{g}(x',x'')= \sum\limits_{k=1}^{q}\int_\gamma\left|\langle \nabla g_{k}^{H,\beta}(x),dx\rangle \right |~$ was defined and analyzed in \cite{KO}.

\bigskip
\noindent
From Theorem \ref{LD_BCWP} and  Lemma \ref{beta_sq} we have that for $\beta<\beta_s$, the point $(\rho, \rho) \in \mathcal{P}_q\times\mathcal{P}_q$ is the unique equilibrium macrostate. Thus, utilizing the aggregate path coupling method developed in \cite{KO} and applied there to the Generalized  Curie-Weiss-Potts model, we have the next proposition that follows immediately from Lemma 10.4 in \cite{KO}.
\begin{prop}\label{condKO_3}
	Suppose $\beta < \beta_s(q)$ and let $(\rho, \rho)$ be the unique equilibrium macrostate. Then
	$$\limsup_{(x,y)\to (\rho,\rho)}\max\left\{
	\frac{\parallel g^{H,\beta}(x)- g^{H,\beta}(\rho) \parallel_1 }{\parallel x-\rho\parallel_1 }, 
	\frac{\parallel g^{H,\beta}(y)- g^{H,\beta}(\rho) \parallel_1}{\parallel y-\rho\parallel_1} \right\}  < 1$$
\end{prop}
\bp
By Lemma \ref{beta_sq} we have that  $\beta_s(q)$ is equal to the mixing time critical value for  the Curie-Weiss-Potts model. As
$g^{H,\beta}(x)$ is the same function for the  Curie-Weiss-Potts model, and $\rho=\left(\frac{1}{q},\dots,
\frac{1}{q}\right)$, employing Lemma 10.4 of \cite{KO} we conclude the proof of the Proposition \ref{condKO_3}.
\ep

\bigskip
\noindent
We now state and prove the main contraction result for the mean coupling distance where one of the coupled processes starts near the equilibrium. 
\begin{lemma}\label{exp_decay}
Suppose $\beta < \beta_s(q)$. Let $(X,Y)$ be a coupling of the Glauber dynamics of the bipartite Potts model starting in configurations $(\sigma,\tau)$ and $(\sigma',\tau')$, and let $(\rho, \rho)$ be the single equilibrium macrostate of the corresponding canonical ensemble $P_{n, n, \beta}$ defined in (\ref{ensemble_BCWP}). Then there exists an $\alpha >0$ and $\varepsilon'$ small enough such that for $n$ large enough,
$$\E_{\substack{ (\sigma,\tau)\\(\sigma',\tau')}} [d(X,Y)] \leq e^{-\alpha/n}d\big((\sigma,\tau),(\sigma',\tau')\big),$$
whenever $\parallel \big(L_n(\sigma),L_n(\tau)\big) - (\rho, \rho)\parallel_1 < \varepsilon'$.
\end{lemma}
\bp
We will follow the steps in the proof of Lemma 9.1 in \cite{KO}. 
Let $\beta<\beta_s$. 

{\bf Case I.} Suppose $\ve>0$ is sufficiently small. Consider a pair of configurations $(\sigma,\tau)$ and $(\sigma',\tau')$ with magnetizations 
$$\big(L_n(\sigma),L_n(\tau)\big)=(z,w) \quad \text{ and } \quad \big(L_n(\sigma'),L_n(\tau')\big)=(z',w').$$  
Also, consider a continuous monotone path $(\gamma, \widetilde{\gamma})$ connecting $(\rho,\rho)$ to $(z',w')$, defined as follows
$$\gamma=\{x(t)=(1-t)\rho  + tz' ~:~ t\in[0,1]\} \quad \text{ and }\quad  \widetilde{\gamma} =\{y(t)=(1-t)\rho  + tw' ~:~ t\in[0,1]\}.$$
Then, as it was proved in Lemma 10.3  of \cite{KO}, there exists $\delta \in (0,1)$ such that 
$${D_{(\gamma, \widetilde{\gamma})}^{g}((x,y),(\rho,\rho)) \over \parallel x-\rho\parallel_1+\parallel y-\rho\parallel_1}
={D_\gamma^g(x,\rho)+D_{\widetilde{\gamma}}^g(y,\rho) \over \parallel x-\rho\parallel_1+\parallel y-\rho\parallel_1}\leq 1-\frac{\delta}{2}.$$
Let $\varepsilon'=\varepsilon^2\delta/M$ for $M$ sufficiently large.  Suppose that $\|(z',w')- (\rho, \rho)\|_1 \geq \ve +\ve'$ and $\|(z,w)- (\rho, \rho)\|_1< \ve'$.

\medskip
\noindent
Thus, provided $\ve$ is sufficiently small and $M$ sufficiently large, for all $n$ large enough, there is a discrete {\it monotone path} in $\mathcal{P}_n\times\mathcal{P}_n$,
$$(z,w)=(z_0,w_0),(z_1,w_1),\dots,(z_r,w_r)= (z',w'),$$ 
approximating (dotting) the continuous monotone path $(\gamma,\widetilde{\gamma})$, such that
$$\ve ~\leq~ \parallel z_i-z_{i-1} \parallel_1 + \parallel w_i-w_{i-1} \parallel_1 ~<~ 2\varepsilon\quad \mbox{for}\quad i=1,2,\dots,r,$$
for which 
\be \label{delta3}
\frac{\dstyle\sum_{k=1}^{q}\dstyle\sum_{i=1}^{r}\left| \left\langle z_i-z_{i-1},\nabla g_k^{H,\beta}(z_{i-1}) \right\rangle \right|+\dstyle\sum_{k=1}^{q}\dstyle\sum_{i=1}^{r}\left| \left\langle w_i-w_{i-1},\nabla g_k^{H,\beta}(w_{i-1}) \right\rangle \right|}{\parallel z'-z \parallel_1 +\parallel w'-w \parallel_1}< 1-\frac{\delta}{3}.
\ee

\medskip
\noindent
Next, one can construct a monotone path as in Definition \ref{mono_path}
$$\pi : (\sigma,\tau)=(x_0^1,x_0^2),(x_1^1,x_1^2),\dots,(x_r^1,x_r^2)=(\sigma',\tau').$$ 
connecting configurations $(\sigma,\tau)$ and $(\sigma',\tau')$ such that
$$\big(L_n(x_i^1),L_n(x_i^2)\big)=(z_i,w_i).$$

\bigskip
\noindent
Hence, by equation (\ref{bound_dist}), 
\begin{align*}
\E_{\substack{ (\sigma,\tau)\\(\sigma',\tau')}} [d(X,Y)] & \leq d\big((\sigma,\tau),(\sigma',\tau')\big)\left[ 1-\dstyle\frac{1}{2n}\left( 1-\frac{ S_1+S_2 + 4c\varepsilon^2}{\parallel L_n(\sigma) - L_n(\sigma') \parallel_1 + \parallel L_n(\tau) - L_n(\tau') \parallel_1} \right) \right]\\
  &\leq d((\sigma,\sigma'),(\tau,\tau'))\left[ 1 - \dstyle\frac{\delta/3-\delta/12}{2n} \right]=
  d((\sigma,\sigma'),(\tau,\tau'))\left[ 1 - \dstyle\frac{\delta}{8n} \right]\\
\end{align*}
as $~{4c\varepsilon^2 \over  \|(z',w')-(z,w)\|_1 } ~\leq ~4c\ve ~\leq ~\delta/12~$ for 
$\ve$ small enough,
where we used the same notation as earlier, 
$$S_1=\dstyle\sum_{k=1}^{q}\dstyle\sum_{i=1}^{r}\left| \left\langle L_n(x_i^1)-L_n(x_{i-1}^1),\nabla g_k^{H,\beta}(L_n(x_{i-1}^1)) \right\rangle \right|$$
and
$$S_2= \dstyle\sum_{k=1}^{q}\dstyle\sum_{i=1}^{r}\left| \left\langle L_n(x_i^2)-L_n(x_{i-1}^2),\nabla g_k^{H,\beta}(L_n(x_{i-1}^2)) \right\rangle \right|.$$

\bigskip
\noindent
{\bf Case II.} Let $\ve$ and $\ve'$ be as in Case I. Suppose $\big(L_n(\sigma),L_n(\tau)\big)=(z,w)$ and $\big(L_n(\sigma'),L_n(\tau')\big)=(z',w')$  such that $\|(z',w')- (\rho, \rho)\|_1< \ve +\ve'$ and $\|(z,w)- (\rho, \rho)\|_1< \ve'$.

\medskip
\noindent
Then, similarly to (\ref{dist_1step}), equation (\ref{eq:k}) and Proposition \ref{condKO_3} imply there is a $\xi>0$ such that for all $n$ large enough,
{\footnotesize
\beas
\E_{\substack{ (\sigma,\tau)\\(\sigma',\tau')}} [d(X,Y)] &  \leq & d((\sigma,\sigma'),(\tau,\tau')) \cdot \left[1-{1 \over 2n}\left(1-{\|g^{H,\beta}\big(L_n(\sigma),L_n(\tau)\big)-g^{H,\beta}\big(L_n(\sigma'),L_n(\tau')\big)\|_1 \over \parallel L_n(\sigma) - L_n(\sigma') \parallel_1 + \parallel L_n(\tau) - L_n(\tau') \parallel_1} \right)  \right] + O\left( {1 \over n^2}\right)\\ 
&  \leq & d((\sigma,\sigma'),(\tau,\tau')) \cdot \left[1-{\xi \over n} \right] + O\left( {1 \over n^2}\right)\\
&  \leq & d((\sigma,\sigma'),(\tau,\tau')) \cdot \left[1-{\xi \over 2n} \right]. 
\eeas
}

\medskip
\noindent 
This, concludes  the proof of Lemma \ref{exp_decay}.
\ep
To conclude, Lemma \ref{exp_decay} above and Theorem 9.2 in \cite{KO} yield the proof to our main result Theorem \ref{rap_mixBCWP}.  Let us list the main steps of the proof as applied in \cite{KO}. 
First, recall that in order to bound a total variation distance between a Markov chain $X_t$ and its stationary distribution via the coupling inequality it is sufficient to couple two copies of the Markov chain,
$X_t$ and $Y_t$, where $Y_0$ is distributed according to the stationary distribution and $X_0$ could be any state. See \cite{AF, LPW, L}. Thus, for a given $\beta < \beta_s(q)$, 
we couple two copies $X_t=(X_t^1,X_t^2)$ and $Y_t=(Y_t^1,Y_t^2)$ of the Markov chain defined in Section \ref{sec:gd}, where $Y_0$ and therefore $Y_t$ for every value $t \geq 0$ is distributed according to the stationary distribution $P_{n,n,\beta}$ over the state space $\Lambda_n \times \Lambda_n$. 

Now, it is known from Lemma \ref{beta_sq} that $\beta_s(q) \leq \beta_c(q)$. Thus, $\beta < \beta_s(q)$ implies that $\mathcal{E}_\beta=\{ (\rho,\rho) \}$ by Theorem \ref{LD_BCWP}. 
By the large deviation principle in Theorem \ref{LDP_Knn}, the probability measure $P_{n,n,\beta}$ is concentrated on the configurations $(\sigma,\tau)\in \Lambda_n \times \Lambda_n$ with 
magnetization $\big(L_n(\sigma),L_n(\tau)\big)$ in the neighborhood of $(\rho,\rho)$ in $\mathcal{P}_q \times \mathcal{P}_q$. 
Therefore, the event $\parallel \big(L_n(Y_t^1),L_n(Y_t^2)\big) - (\rho, \rho)\parallel_1 < \varepsilon'$ required for the contraction result in Lemma \ref{exp_decay} is satisfied since the probability of its complement
is exponentially small, i.e.
$$P_{n,n,\beta}\left(\parallel \big(L_n(Y_t^1),L_n(Y_t^2)\big) - (\rho, \rho)\parallel_1 \geq \varepsilon' \right) < e^{-{n \over \xi'}I_\beta(\varepsilon')} \qquad \text{ for }~\xi'>1.$$
Hence, the mean distance between the configurations $X_t=(X_t^1,X_t^2)$ and $Y_t=(Y_t^1,Y_t^2)$ shrinks by the multiple of $e^{-\alpha/n}$ on each time step, and becomes $\ll 1$ after an order of $O(n \log{n})$ time steps.

\medskip
Finally, as we already stated following the statement of Theorem \ref{rap_mixBCWP}, the standard bottleneck ratio argument applying the Cheeger constant proves slow mixing for $\beta > \beta_s(q)$. See \cite{LPW}.



\bibliographystyle{amsplain}

\end{document}